\documentclass{amsart}
\usepackage{graphicx}
\newtheorem{thm}{Theorem}[section]

\theoremstyle{definition}

\theoremstyle{example}

\theoremstyle{remark}
\newtheorem{rem}[thm]{Remark}
\numberwithin{equation}{section}

\begin{document}

\title[Binomial sums]{On some Binomial Coefficient Identities with Applications}%
\author{necdet bat{\i}r}%
\address{department of  mathematics, nev{\c{s}}ehir hbv university, nev{\c{s}}ehir, 50300 turkey}%
\email{nbatir@hotmail.com}%
\author{Sezer Sorgun}
\address{department of  mathematics, nev{\c{s}}ehir hbv university, nev{\c{s}}ehir, 50300 turkey}
\email{ srgrnzs@gmail.com}
\author{Sevda Atp{\i}nar}
\address{department of  mathematics, nev{\c{s}}ehir hbv university, nev{\c{s}}ehir, 50300 turkey}%
\email{sevdaatpinar@nevsehir.edu.tr}%


\begin{abstract}
We present a different proof of the following identity due to Munarini, which generalizes a curious binomial identity of Simons.
\begin{align*}
\sum_{k=0}^{n}\binom{\alpha}{n-k}\binom{\beta+k}{k}x^k
&=\sum_{k=0}^{n}(-1)^{n+k}\binom{\beta-\alpha+n}{n-k}\binom{\beta+k}{k}(x+1)^k,
\end{align*}
where $n$ is a non-negative integer and $\alpha$ and $\beta$ are complex numbers, which are not negative integers. Our approach is based on a particularly  interesting combination of the Taylor theorem and the Wilf-Zeilberger algorithm. We also generalize  a combinatorial identity due to Alzer and Kouba, and offer a new binomial sum identity.  Furthermore, as applications,  we give many harmonic number sum identities. As examples, we prove that
\begin{equation*}
H_n=\frac{1}{2}\sum_{k=1}^{n}(-1)^{n+k}\binom{n}{k}\binom{n+k}{k}H_k
\end{equation*}
and
\begin{align*}
\sum_{k=0}^{n}\binom{n}{k}^2H_kH_{n-k}=\binom{2n}{n} \left((H_{2n}-2H_n)^2+H_{n}^{(2)}-H_{2n}^{(2)}\right).
\end{align*}
\end{abstract}
\maketitle
\section{Introduction}\label{Sec1}
In 2001 Simons \cite{Simons} discovered a curious binomial identity which can be equivalently written as
\begin{equation}\label{eq1}
\sum_{k=0}^{n}\binom{n}{k}\binom{n+k}{k}x^k=\sum_{k=0}^{n}(-1)^{n+k}\binom{n}{k}\binom{n+k}{k}(x+1)^k.
\end{equation}
This formula has attracted  the attention of several researchers and many  alternative proofs of it have appeared in the literature. These authors used  a variety of techniques in their proofs. For instance, Bat{\i}r and Atp{\i}nar  \cite{Batir3} extended \eqref{eq1} using the Taylor theorem. Chapman \cite{Chapman}, using the generating function method, gave an elegant and short proof of \eqref{eq1}. Prodinger \cite{Prodinger} used  the Cauchy integral formula to demonstrate a novel and very short proof of this identity. Wang and Sun \cite{Wang} presented a very short proof of it  utilizing  a linear transformation in the space of polynomials. Shattuck \cite{Shattuck}  presented combinatorial proofs of some Simons-type binomial coefficient identities. Hirschhorn \cite{Hirschhorn} made short comments on Simons' curious identity. Also, see some recent studies given by Wei et al. \cite{Wei} and  Xu and Cen \cite{Xu}. Munarini \cite{Munarini}, using the Cauchy integral formula, offered the following nice generalization of (\ref{eq1}):
 \begin{align}\label{eq2}
\sum_{k=0}^{n}\binom{\alpha}{n-k}&\binom{\beta+k}{k}x^ky^{n-k}\notag\\
&=\sum_{k=0}^{n}(-1)^{n+k}\binom{\beta-\alpha+n}{n-k}\binom{\beta+k}{k}(x+y)^ky^{n-k}.
\end{align}
This can be rewritten as follows after replacing $x$ by $\frac{x}{y}$.
\begin{align}\label{eq3}
\sum_{k=0}^{n}\binom{\alpha}{n-k}\binom{\beta+k}{k}x^k=\sum_{k=0}^{n}(-1)^{n+k}\binom{\beta-\alpha+n}{n-k}\binom{\beta+k}{k}(x+1)^k.
\end{align}
Note that  (\ref{eq2}) reduces to (\ref{eq1}) when we put $\alpha=\beta=n$ and $y=1$. Our first aim in the paper is to provide a different proof of \eqref{eq2} using the Taylor theorem and the Wilf-Zeilberger algorithm; see \cite{Petkovsek}.

In a recent paper \cite[Eq. (4.1)]{Alzer}), Alzer and Kouba proved the following identity
\begin{equation}\label{Al}
4^n\binom{\lambda}{n}=\binom{2\lambda}{n}\sum_{k=0}^{n}\binom{n}{k}\frac{\binom{n-\lambda-1/2}{k}}{\binom{k-\lambda-1/2}{k}},
\end{equation}
which is valid  for all integers $n\geq 0$ and  any complex number $\lambda$, and they used it to derive several harmonic number sum identities.  The second aim of this work is to generalize identity  \eqref{Al} and to apply it to deduce further binomial sum identities involving harmonic numbers. Our final aim is to prove a  new binomial sum identity, which is stated in  Theorem \ref{thm2}. The proofs of Theorems \ref{thm2} and \ref{thm3} are relied on  the WZ theory. Applications of our results lead to many interesting and new combinatorial identities involving harmonic numbers. We also give a new representation for the Legendre polynomials $P_n(x)$. Many examples of these kinds obtained by taking particular values of the parameters involved are collected  in Section \ref{Sec3}. To ensure accuracy, all formulas appearing in this paper were verified numerically by \textit{Mathematica}. We believe that the method we used in the proofs of the theorems can be applied to prove some other combinatorial equalities.

Now we need to recall some basic tools that we use  to prove our examples. The classical gamma  function $ \Gamma(x)=\int_{0}^{\infty }u^{x-1}e^{-u}du$ and the digamma function $\psi $, logarithmic derivative of the $\Gamma$-function satisfy the following relations for any integer $n\geq 0$
\begin{equation}\label{eq4}
\Gamma\left(n+\frac{1}{2}\right)=\frac{(2n)!\Gamma(1/2)}{2^{2n}n!},
\end{equation}
and for any complex number $x$, which is not an integer
\begin{equation}\label{eq5}
\Gamma(x)\Gamma(1-x)=\frac{\pi}{\sin(\pi x)},
\end{equation}
\begin{equation}\label{eq6}
\psi (n+1/2)-\psi(1/2)=2H_{2n}-H_n,
\end{equation}
where  $\gamma=0.57721...$ is the Euler-Mascheroni constant and $H_n=1+\frac{1}{2}+\frac{1}{3}+\cdots+\frac{1}{n}$ is $n$ \textit{th} harmonic number, and
\begin{equation}\label{eq7}
\psi (1-x)-\psi (x)=\pi \cot (\pi x)\quad \mbox{and}\quad \psi ^{\prime
}\left( x\right) -\psi ^{\prime }\left( 1-x\right) =\pi ^{2}\csc ^{2}(\pi x);
\end{equation}
see, for example, \cite{Srivastava}. A generalized binomial coefficient $\binom{s}{t}$ is defined, in terms of the classical gamma function, by
\begin{equation}\label{eq8}
\binom{s}{t}=\frac{\Gamma (s+1)}{\Gamma (t+1)\Gamma (s-t+1)},\quad s,t\in
\mathbb{C}.
\end{equation}

\section{Main Results}\label{Sec2}
In the first theorem in this section we give a different proof of \eqref{eq3} using the Taylor theorem and the WZ theory.
\begin{thm}\label{thm1}For all $\alpha,\beta\in \mathbb{C}$, which are not negative integers,  and any non-negative integer $n$ we have
\begin{align}\label{eq9}
\sum_{k=0}^{n}\binom{\alpha}{n-k}\binom{\beta+k}{k}x^k=\sum_{j=0}^{n}(-1)^{n+j}\binom{\beta-\alpha+n}{n-j}\binom{\beta+j}{j}(x+1)^j.
\end{align}
\end{thm}
\begin{proof}
We define
\begin{align}\label{eq10}
f(x)=\sum_{k=0}^{n}\binom{\alpha}{n-k}\binom{\beta+k}{k}x^k.
\end{align}
Since $f$ is a polynomial of degree $n$ in $x$ its Taylor polynomial at $x=-1$ is
\begin{equation}\label{eq11}
f(x)=\sum_{j=0}^{n}\frac{f^{(j)}(-1)}{j!}(x+1)^j.
\end{equation}
Differentiating $f$ $j$ times it follows that
\begin{align*}
\frac{f^{(j)}(x)}{j!}=\sum_{k=0}^{n}\binom{\alpha}{n-k}\binom{\beta+k}{k}\binom{k}{j}x^{k-j}.
\end{align*}
We set $x=-1$ and get
\begin{align*}
\frac{f^{(j)}(-1)}{j!}=\sum_{k=0}^{n}\binom{\alpha}{n-k}\binom{\beta+k}{k}\binom{k}{j}(-1)^{k-j}.
\end{align*}
Applying this to \eqref{eq11} yields
\begin{equation*}
f(x)=\sum_{j=0}^{n}\left(\sum_{k=0}^{n}\binom{\alpha}{n-k}\binom{\beta+k}{k}\binom{k}{j}(-1)^{k-j}\right)(x+1)^j.
\end{equation*}
Comparing the coefficients of $(x+1)^j$  in this identity and the right-hand side of \eqref{eq9} we conclude that \eqref{eq9} holds if and only if  for all $\alpha,\beta\in \mathbb{C}$, which are not negative integers, and all  $j=0,1,2,..., n$ we have
\begin{equation}\label{eq12}
\sum_{k=0}^{n}(-1)^{k+j} \binom{\beta+k}{k} \binom{k}{j} \binom{\alpha}{n-k}=(-1)^{n+j}\binom{\beta+j}{j}\binom{\beta-\alpha+n}{n-j}.
\end{equation}
In order to show that \eqref{eq12} is valid we define
\begin{equation}\label{eq13}
F(n,k)=\frac{(-1)^{k+n} \binom{\beta+k}{k} \binom{k}{j} \binom{\alpha}{n-k}}{\binom{\beta+j}{j} \binom{-\alpha+\beta+n}{n-j}}.
\end{equation}
Applying the Wilf-Zeilberger algorithm \cite{Petkovsek}, we find the companion function
\begin{equation}\label{eq14}
G(n,k)=\frac{(-1)^{n+k}(j-k) (\alpha+k-n) \binom{\beta+k}{k} \binom{k}{j} \binom{\alpha}{n-k}}{(k-n-1) (\alpha-\beta-n-1) \binom{\beta+j}{j} \binom{-\alpha+\beta+n}{n-j}}
\end{equation}
and  the following difference equation holds
\begin{equation*}
F(n,k)-F(n+1,k)=G(n,k+1)-G(n,k).
\end{equation*}
We sum both  sides of this equation from  $k=0$ to $k=n+1$  and get
\begin{equation*}
\sum_{k=0}^{n+1}\big(F(n,k)-F(n+1,k)\big)=\sum_{k=0}^{n+1}\big(G(n,k+1)-G(n,k)\big).
\end{equation*}
The right hand side of this equality is a telescoping sum, thus
\begin{equation*}
\sum_{k=0}^{n+1}F(n,k)-\sum_{k=0}^{n+1}F(n+1,k)=G(n,n+2)-G(n,0).
\end{equation*}
It is obvious  from (15) that $G(n,0)=G(n,n+2)=0$. Note that $\binom{\alpha}{-2}=0$. This leads to
\begin{equation*}
\sum_{k=0}^{n+1}F(n,k)-\sum_{k=0}^{n+1}F(n+1,k)=0.
\end{equation*}
or equivalently
\begin{equation*}
\sum_{k=0}^{n+1}F(n+1,k)-\sum_{k=0}^{n}F(n,k)-F(n,n+1)=0.
\end{equation*}
But since  $F(n,n+1)=0$, which follows from \eqref{eq13}, this gives
\begin{equation*}
\sum_{k=0}^{n+1}F(n+1,k)-\sum_{k=0}^{n}F(n,k)=0.
\end{equation*}
Replacing $n$ by $\upsilon$  here and then summing the resultant identity from $\upsilon=0$ to $\upsilon=n-1$  we conclude
\begin{equation*}
\sum_{\upsilon=0}^{n-1}\bigg(\sum_{k=0}^{\upsilon+1}F(\upsilon+1,k)-\sum_{k=0}^{\upsilon}F(\upsilon,k)\bigg)=0.
\end{equation*}
This is also a telescoping sum, thus, we obtain
\begin{equation*}
\sum_{k=0}^{n}F(n,k)-F(0,0)=0.
\end{equation*}
Since $0\leq j\leq n$  we have from \eqref{eq13} that $F(0,0)=1$, so that
\begin{equation*}
\sum_{k=0}^{n}F(n,k)=1,
\end{equation*}
which is equivalent to  \eqref{eq12}. This completes the proof of Theorem \ref{thm1}.
\end{proof}
In the next theorem, we  generalize the identity given by \eqref{Al}.
\begin{thm}\label{thm2} Let $s$ and $t$ be complex numbers, which are not negative integers and $n$ be a non-negative integer. Then we have
\begin{equation}\label{eq15}
\sum_{k=0}^{n}\binom{n}{k}\frac{\binom{s}{k}}{\binom{t+k}{k}}=\frac{\binom{n+s+t}{s}}{\binom{s+t}{s}}.
\end{equation}
\end{thm}
\begin{proof} Our proof is based on the Wilf-Zeilberger algorithm. Let
\begin{equation}\label{eq16}
A(n,k)=\frac{\binom{n}{k}\binom{s}{k}\binom{s+t}{t}}{\binom{t+k}{k}\binom{n+s+t}{s}}.
\end{equation}
In order to complete the proof it suffices to show that
$$
\sum_{k=0}^{n}A(n,k)=1.
$$
The Wilf--Zeilberger algorithm gives us
\begin{equation}\label{eq17}
B(n,k)=\frac{k(t+k)\binom{n}{k}\binom{s}{k}\binom{t+s}{t}}{(k-n-1)(n+t+s+1)\binom{t+k}{k}\binom{n+s+t}{s}}
\end{equation}
and the pair $(A,B)$ is a  WZ-pair. That is,
\begin{equation*}
A(n+1,k)-A(n,k)=B(n,k+1)-B(n,k).
\end{equation*}
We sum both  sides of this equation from  $k=0$ to $k=n+1$.  This gives us
\begin{equation*}
\sum_{k=0}^{n+1}\big(A(n+1,k)-A(n,k)\big)=\sum_{k=0}^{n+1}\big(B(n,k+1)-B(n,k)\big).
\end{equation*}
The right hand side of this equality is a telescoping sum, thus
\begin{equation*}
\sum_{k=0}^{n+1}A(n+1,k)-\sum_{k=0}^{n+1}A(n,k)=B(n,n+2)-B(n,0).
\end{equation*}
It is obvious  from \eqref{eq17} that $B(n,0)=B(n,n+2)=0$. This leads to
\begin{equation*}
\sum_{k=0}^{n+1}A(n+1,k)-\sum_{k=0}^{n+1}A(n,k)=0.
\end{equation*}
or equivalently
\begin{equation*}
\sum_{k=0}^{n+1}A(n+1,k)-\sum_{k=0}^{n}A(n,k)-A(n,n+1)=0.
\end{equation*}
But since  $A(n,n+1)=0$, which follows from \eqref{eq16}, it follows that
\begin{equation*}
\sum_{k=0}^{n+1}A(n+1,k)-\sum_{k=0}^{n}A(n,k)=0.
\end{equation*}
Replacing $n$ by $p$  here and then summing the resultant identity from $p=0$ to $p=n-1$  we arrive at
\begin{equation*}
\sum_{p=0}^{n-1}\bigg(\sum_{k=0}^{p+1}A(p+1,k)-\sum_{k=0}^{p}A(p,k)\bigg)=0.
\end{equation*}
This is also a telescoping sum, thus, we obtain
\begin{equation*}
\sum_{k=0}^{n}A(n,k)-A(0,0)=0.
\end{equation*}
By \eqref{eq16} we have $A(0,0)=1$, so that
\begin{equation*}
\sum_{k=0}^{n}A(n,k)=1,
\end{equation*}
completing the proof of Theorem \ref{thm2}.
\end{proof}
Note that \eqref{eq15}  reduces to \eqref{Al} when  we substitute $s=n-\lambda-\frac{1}{2}$ and $t=-\lambda-\frac{1}{2}$.
\begin{thm}\label{thm3} Let $s$ and $p$ be complex numbers, which are not negative integers, and $n$ be a non-negative integer. Then we have
\begin{equation}\label{eq18}
\sum_{k=0}^{n}(-1)^{n+k}\binom{n}{k}\binom{s+k}{k}\binom{k}{p}=\binom{n}{p}\binom{s+p}{n}.
\end{equation}
\end{thm}
\begin{proof} We again employ the WZ theory. Let
\begin{equation}\label{eq19}
P(n,k)=\frac{(-1)^{n+k}\binom{n}{k}\binom{s+k}{k}\binom{k}{p}}{\binom{n}{p}\binom{s+p}{n}}.
\end{equation}
 Using the Wilf-Zeilberger algorithm, this gives us the companion function
\begin{equation}\label{eq20}
Q(n,k)=\frac{(-1)^{n+k}k(k-p)\binom{n}{k}\binom{k}{p}\binom{-1-s}{k}}{(-1+k-n)(n-p-s)\binom{n}{p}\binom{s+p}{n}},
\end{equation}
and we obtain that the pair $(P,Q)$ satisfies the following difference equation
\begin{equation*}
P(n,k)-P(n+1,k)=Q(n,k+1)-Q(n,k).
\end{equation*}
We sum both  sides of this equation from  $k=0$ to $k=n+1$  and get
\begin{equation*}
\sum_{k=0}^{n+1}\big(P(n,k)-P(n+1,k)\big)=\sum_{k=0}^{n+1}\big(Q(n,k+1)-Q(n,k)\big).
\end{equation*}
The right hand side of this equality is a telescoping sum, so that
\begin{equation*}
\sum_{k=0}^{n+1}P(n,k)-\sum_{k=0}^{n+1}P(n+1,k)=Q(n,n+2)-Q(n,0).
\end{equation*}
It is clear  from \eqref{eq20} that $Q(n,0)=Q(n,n+2)=0$. This leads to
\begin{equation*}
\sum_{k=0}^{n+1}P(n,k)-\sum_{k=0}^{n+1}P(n+1,k)=0
\end{equation*}
or equivalently
\begin{equation*}
\sum_{k=0}^{n+1}P(n+1,k)-\sum_{k=0}^{n}P(n,k)-P(n,n+1)=0.
\end{equation*}
But since  $P(n,n+1)=0$, which follows from \eqref{eq19}, this gives
\begin{equation*}
\sum_{k=0}^{n+1}P(n+1,k)-\sum_{k=0}^{n}P(n,k)=0.
\end{equation*}
Replacing $n$ by $p$  here and then summing the resultant identity from $p=0$ to $p=n-1$  we obtain
\begin{equation*}
\sum_{p=0}^{n-1}\bigg(\sum_{k=0}^{p+1}P(p+1,k)-\sum_{k=0}^{p}P(p,k)\bigg)=0.
\end{equation*}
This is also a telescoping sum, thus, we have
\begin{equation*}
\sum_{k=0}^{n}P(n,k)-P(0,0)=0.
\end{equation*}
By \eqref{eq19} we have $P(0,0)=1$, so that
\begin{equation*}
\sum_{k=0}^{n}P(n,k)=1,
\end{equation*}
which is equivalent to \eqref{eq18}.
\end{proof}

\section{Applications}\label{Sec3}
In this section we provide many applications of Theorems \ref{thm1}, \ref{thm2} and \ref{thm3} by setting particular values for  the paremeters $\alpha,\,\beta$, $t$, $s$ and $p$. It seems to us  that all examples offered here are new in the literature. If we set $\alpha=n$ in \eqref{eq9} and use
$$
\binom{\beta}{n-k}\binom{\beta+k}{k}=\binom{n}{k}\binom{\beta+k}{n},
$$
which can be easily verified, we get
\begin{equation}\label{eq21}
\sum_{k=0}^{n}\binom{n}{k}\binom{\beta+k}{k}x^k=\sum_{k=0}^{n}(-1)^{n+k}\binom{n}{k}\binom{\beta+k}{n}(1+x)^k,
\end{equation}
where $\beta$ and $x$ are complex numbers with $\beta\neq -1,-2,-3,...$  and $n\geq 0$ be an integer.\\
$\bullet$ We set  $x=0$ and $x=-1$ in \eqref{eq21} and get, respectively
\begin{equation}\label{eq22}
\sum_{k=0}^{n}(-1)^{k}\binom{n}{k}\binom{\beta+k}{n}=(-1)^n
\end{equation}
and
\begin{equation*}
\sum_{k=0}^{n}(-1)^k\binom{n}{k}\binom{\beta+k}{k}=(-1)^n\binom{\beta}{n}.
\end{equation*}
$\bullet$ Differentiate both sides of (\ref{eq21}) with respect to $\beta$ and then  put $\beta=n$ to obtain
\begin{equation*}
\sum_{k=0}^{n}(-1)^k\binom{n}{k}\binom{n+k}{k}H_{n+k}-H_n\sum_{k=0}^{n}(-1)^k\binom{n}{k}\binom{n+k}{k}=(-1)^nH_n.
\end{equation*}
Using \eqref{eq22} with $\beta=n$, this gives the following explicit representation for $H_n$.
\begin{equation*}
H_n=\frac{1}{2}\sum_{k=0}^{n}(-1)^{n+k}\binom{n}{k}\binom{n+k}{k}H_{n+k}.
\end{equation*}
$\bullet$ Substituting  $\beta=-\frac{1}{2}$ in \eqref{eq21} produces
\begin{equation}\label{eq23}
\sum_{k=0}^{n}\binom{n}{k}\binom{-\frac{1}{2}+k}{k}x^k=\sum_{k=0}^{n}(-1)^{n+k}\binom{n}{k}\binom{-\frac{1}{2}+k}{n}(1+x)^k.
\end{equation}
We have
\begin{equation*}
\binom{-1/2+k}{k}=\frac{\Gamma(k+1/2)}{k!\Gamma(1/2)}=\frac{1}{4^{k}}\binom{2k}{k}
\end{equation*}
and
\begin{equation*}
 \binom{-1/2+k}{n}=\frac{\Gamma(k+1/2)}{n!\Gamma(1/2+k-n)}=\frac{(-1)^{n+k}\binom{2k}{k}\binom{2n-2k}{n-k}}{4^n\binom{n}{k}},
\end{equation*}
both of which follow from the Legendre's duplication and reflection formulas for the classical gamma function $\Gamma$ given in \eqref{eq4}, \eqref{eq5} and \eqref{eq6}. Replacing these in (\ref{eq23}) it follows that
\begin{equation}\label{eq24}
\sum_{k=0}^{n}\binom{n}{k}\binom{2k}{k}\frac{x^k}{2^{2k}}=\frac{1}{2^{2n}}\sum_{k=0}^{n}\binom{2k}{k}\binom{2n-2k}{n-k}(1+x)^k.
\end{equation}
$\bullet$ Let $P_n(x)$ be the classical Legendre's polynomials. From \cite[p. 39]{Gould} (see also \cite{Alzer}) we have
\begin{equation*}
t^nP_n\left(\frac{t^2+1}{2t}\right)=\frac{1}{4^n}\sum_{k=0}^{n}\binom{2k}{k}\binom{2n-2k}{n-k}t^{2k}.
\end{equation*}
Making the substitution $x\to t^2-1$ in \eqref{eq24}, we have
\begin{equation*}
\sum_{k=0}^{n}\binom{n}{k}\binom{2k}{k}\left(\frac{t^2-1}{4}\right)^k=\frac{1}{4^{n}}\sum_{k=0}^{n}\binom{2k}{k}\binom{2n-2k}{n-k}t^{2k}.
\end{equation*}
From the last two identities we deduce the following new representation for the Legendre's polynomials for $t\neq 0$:
\begin{equation*}
P_n\left(\frac{t^2+1}{2t}\right)=\frac{1}{t^n}\sum_{k=0}^{n}\binom{n}{k}\binom{2k}{k}\left(\frac{t^2-1}{4}\right)^k.
\end{equation*}
By the well-known binomial inversion formula this leads to
\begin{equation*}
\sum_{k=0}^{n}(-1)^k\binom{n}{k}P_k\left(\frac{t^2+1}{2t}\right)t^k =\binom{2n}{n}\left(\frac{1-t^2}{4}\right)^n.
\end{equation*}
$\bullet$ Differentiating both sides of \eqref{eq18} with respect to $p$ and then setting $p=0$, we get
\begin{equation}\label{eq25}
\sum_{k=0}^{n}(-1)^{n+k}\binom{n}{k}\binom{s+k}{k}H_k=\binom{s}{n}\big(H_n+\psi(s+1)-\psi(s-n+1)\big).
\end{equation}
For $s=n$ this leads to the following elegant representation for $H_n$:
\begin{equation*}
H_n=\frac{1}{2}\sum_{k=1}^{n}(-1)^{n+k}\binom{n}{k}\binom{n+k}{k}H_k,
\end{equation*}
where we have used $\psi(n+1)+\gamma=H_n$. Setting $s=\frac{1}{2}$ in \eqref{eq25} we get
\begin{equation*}
\sum_{k=1}^{n}(-1)^{k}\binom{n}{k}\binom{2k}{k}\frac{H_k}{4^k}=\frac{1}{2^{2n-1}}\binom{2n}{n}\big(H_n-H_{2n}\big).
\end{equation*}
$\bullet$ Differentiating twice both sides of \eqref{eq18} with respect to $p$ and then setting $p=0$ we find that
\begin{align*}
&\sum_{k=0}^{n}(-1)^{n+k}\binom{n}{k}\binom{s+k}{k}\big(H_k^2+H_n^{(2)}\big)\\
&=\binom{s}{n} \bigg(2H_n\big(\psi(s+1)-\psi(-n+s+1)\big)+\psi(-n+s+1)^2\\
&-2 \psi(s+1)\psi(-n+s+1)-\psi ^{\prime}(-n+s+1)+H_n^2\\
&+H_n^{(2)}+\psi(s+1)^2+\psi ^{\prime}(s+1)-\psi^\prime(s-n+1)\bigg)
\end{align*}
where  $\psi$ is the digamma function. For $s=n$ this gives
\begin{align*}
&H_n^2=\frac{1}{4}\sum_{k=0}^{n}(-1)^{n+k}\binom{n}{k}\binom{n+k}{k}\big(H_k^2+H_k^{(2)}\big),
\end{align*}
where $H_k^{(2)}=1+\frac{1}{4}+\frac{1}{9}+\cdots+\frac{1}{k^2}$.\\
$\bullet$ Applying the binomial inversion formula to \eqref{eq18} we get
\begin{equation*}
\sum_{k=0}^{n}\binom{n}{k}\binom{s+p}{k}\binom{k}{p}=\binom{n}{p}\binom{s+n}{n}.
\end{equation*}
Setting $s=n-\frac{1}{2}$ and $p=\frac{1}{2}$ here we obtain
\begin{equation}\label{eq26}
\sum_{k=0}^{n}\frac{4^k\binom{n}{k}^2}{\binom{2k}{k}}=\frac{\binom{4n}{2n}}{\binom{2n}{n}}.
\end{equation}
Using
$$
\binom{2n}{n}\binom{n}{k}^2=\binom{2k}{k}\binom{2n-2k}{n-k}\binom{2n}{2k}
$$
\eqref{eq26} can be rewritten in the following elegant form
\begin{equation*}
\sum_{k=0}^{n}\binom{2n}{2k}\binom{2n-2k}{n-k}4^k=\binom{4n}{2n}.
\end{equation*}
$\bullet$ Setting $p=n+\frac{1}{2}$ in \eqref{eq15} we get
\begin{equation}\label{eq27}
\sum_{k=0}^{n}\binom{s+k}{k}\binom{2n-2k}{n-k}4^k=\frac{\binom{2n}{n}\binom{2n+2s+1}{2s+1}}{\binom{n+s}{n}}.
\end{equation}
If we differentiate both sides of \eqref{eq27} with respect to $s$ and then set $s=0$ we conclude that
\begin{equation*}
\sum_{k=0}^{n}\binom{2k}{k}\frac{H_{n-k}}{4^k}=\frac{2n+1}{4^n}\binom{2n}{n}\big(2H_{2n+1}-H_n-2\big).
\end{equation*}
$\bullet$ From \eqref{eq15} with $s=n$ and $t=1$ we obtain
\begin{equation*}
\sum_{k=1}^{n}k\binom{n}{k}^2=\frac{n}{2}\binom{2n}{n}.
\end{equation*}
Applying the operator $\frac{\partial^2}{\partial s\partial t}$ both sides of \eqref{eq15}, using
$$
\sum_{k=0}^{n}\binom{n}{k}^2H_k=\binom{2n}{n}\left(2H_{n}-H_{2n}\right),
$$
and finally putting $t=0$ and $s=n$ in the resultant we get
\begin{align*}
\sum_{k=0}^{n}\binom{n}{k}^2H_kH_{n-k}&=\binom{2n}{n} \left((H_{2n}-2H_n)^2+H_{n}^{(2)}-H_{2n}^{(2)}\right).
\end{align*}
We differentiate both sides of \eqref{eq18} with respect to $t$ twice and then we  set $s=n$ and $t=0$. This gives
\begin{align*}
\sum_{k=0}^{n}\binom{n}{k}^2\left((H_k)^2+H_k^{(2)}\right)=\binom{2n}{n}\left((H_{2n}-2H_n)^2+2H_n^{(2)}-H_{2n}^{(2)}\right).
\end{align*}

\begin{rem}It is possible to derive many other examples using our main results, but we are satisfied with these to keep our paper short. Clearly, the examples presented  here demonstrate the usefulness of our main results. Several other applications of \eqref{eq21} can be found in \cite{Batir3}.
\end{rem}
\bibliographystyle{amsplain}
{}
\end{document}